\documentclass{cmfnre}
\usepackage[cp1251]{inputenc}
\usepackage[russian]{babel}
\usepackage{cite}

\title[Об успокоении системы управления нейтрального типа на временнoм графе-звезде]{Об успокоении системы управления нейтрального типа на временнoм графе-звезде с глобальным запаздыванием, пропорциональным времени} 
\subjclass{517.9} 
\author{А.П. Леднов} 
\address{Саратовский национальный исследовательский государственный университет имени Н.Г. Чернышевского, Саратов, Россия} 
\address{Московский центр фундаментальной и прикладной математики, Москва, Россия} 
\address{Московский государственный университет имени М.В. Ломоносова, Москва, Россия} 
\email{lednovalexsandr@gmail.com} 

\thanks{Работа выполнена при поддержке РНФ, проект №~24-71-10003.}

\theoremstyle{plain} 
\newtheorem{theorem}{Теорема}[section]
\newtheorem{lemma}{Лемма}[section] 
\theoremstyle{definition} 
\newtheorem{definition}{Определение}[section]
\newtheorem{remark}{Замечание}[section]
\newtheorem{example}{Пример}[section]

\numberwithin{equation}{section}

\begin{document}
\begin{abstract}
На временн\'ом графе-звезде рассматривается задача об оптимальном успокоении системы управления для обобщенного уравнения пантографа, представляющего собой уравнение нейтрального типа с запаздыванием, пропорциональным времени.
Запаздывание в системе распространяется через внутреннюю вершину графа.
Исследуется вариационная задача минимизации функционала энергии с учетом вероятностей сценариев, соответствующих различным ребрам.
Установлено, что оптимальная траектория удовлетворяет условиям типа Кирхгофа во внутренней вершине.
Доказана эквивалентность вариационной задачи некоторой краевой задаче для функционально-дифференциальных уравнений второго порядка на графе и установлена однозначная разрешимость обеих задач.
\end{abstract}
\maketitle
\tableofcontents

\section{Введение} 

Дифференциальные операторы на графах, часто называемые квантовыми графами, активно изучаются с прошлого века в связи с моделированием различных процессов, протекающих в сложных системах, представимых в виде пространственных сетей \cite{7,8,9,10,-10}. Для таких моделей помимо условий непрерывности в вершинах характерны также условия Кирхгофа.

Для задания на графах функционально-дифференциальных операторов с запаздыванием, С.А. Бутериным в работе \cite{5} была предложена концепция глобального запаздывания. Последнее означает, что запаздывание распространяется через внутренние вершины графа. Другими словами,  решение уравнения на входящем ребре служит начальной функцией для уравнений на исходящих ребрах. Глобальное запаздывание стало альтернативой локально нелокальному случаю рассмотренному в \cite{11}, когда уравнение на каждом ребре имеет свой собственный параметр запаздывания и может быть решено отдельно от уравнений на остальных ребрах.

Использование концепции глобального запаздывания позволило перенести на графы класс задач об успокоении управляемых систем с последействием. Впервые задача этого типа была поставлена и исследована на интервале Н. Н. Красовским \cite{1} для системы управления с постоянным запаздыванием, описываемой уравнением запаздывающего типа. Эта задача получила дальнейшее развитие в работах А. Л. Скубачевского \cite{2,14} и позже в работах других авторов (см. \cite{-12} и литературу там), где рассматриваемая система управления имеет нейтральный тип, т.е. содержит запаздывание и в главных членах. Это существенно усложняет задачу и, в частности, приводит к понятию обобщенного решения соответствующей краевой задачи для оптимальной траектории. С.А. Бутерин в работах \cite{3,4} распространил на графы задачу об успокоении систем управления с постоянным запаздыванием. В \cite{3} рассмотрен случай уравнения первого порядка запаздывающего типа, а в \cite{4} -- общий случай нестационарной управляемой системы, уравнения которой относятся к нейтральному типу и имеют произвольный порядок.

Рассмотрение указанной задачи на графах, в свою очередь, привело к концепции временн\'ого графа, ребра которого, в отличие от пространственной сети, отождествляются с промежутками времени, а каждая внутренняя вершина понимается как точка разветвления процесса, дающая несколько различных сценариев дальнейшего его протекания. В \cite{3,4,12,13} показано, что на временн\'ых графах также могут возникать условия Кирхгофа. А именно,  им будет удовлетворять  траектория течения процесса, являющаяся оптимальной с учетом сразу всех сценариев. Кроме того, в \cite{3,12} была предложена стохастическая интерпретация системы управления на временн\'ом дереве. В частности, к системе на дереве приведет замена коэффициентов в уравнении на интервале дискретными случайными процессами с дискретным временем. 

В работах \cite{98,99} на графы была перенесена задача об успокоении системы управления, описываемой так называемым уравнением пантографа \cite{19}. В данном случае запаздывание не постоянно, а является пропорциональным времени сжатием. Рассматривалась система управления, задаваемая классическим уравнением пантографа
\begin{equation}
y^{\prime}(t)+b y(t)+c y(q^{-1} t)=0, \quad t>0, \quad y(0)=y_0,
\label{0.1}
\end{equation}
где $b,c\in\mathbb{R}$, $q>1$. Уравнение вида \eqref{0.1} широко применяется в прикладных задачах. Так, например, это уравнение используется при моделировании динамики контактного провода электроснабжения подвижного состава \cite{20}. Для $q\in(0,1)$ оно возникает в астрофизике при описании поглощения света межзвездной материей \cite{15}, а также в биологии при моделировании процесса роста клеток \cite{16}.

В данной работе мы переходим к рассмотрению системы управления для обобщенного уравнения пантографа
\begin{equation}
y^{\prime}(t)+a y^{\prime}(q^{-1} t)+b y(t)+c y(q^{-1} t)=0, \quad t>0, \quad y(0)=y_0,
\label{0.2}
\end{equation}
где $a, b, c \in \mathbb{R}$, $q>1$. Уравнение вида \eqref{0.2} изучалось в \cite{17,18} и ряде других работ. Были получены различные представления решения и показано, что разрешимость задачи \eqref{0.2} зависит от коэффициента $a$ и от класса гладкости решений. В частности, при $a\neq -q^{k}$, $k\geqslant 0$, существует единственное решение в $C^{\infty}[0,+\infty)$; при этом в зависимости от значения $a$ могут существовать и другие $C^1$-решения, не принадлежащие $C^{\infty}[0,+\infty)$.

На интервале задача об успокоении системы управления для обобщенного уравнения пантографа была рассмотрена Л.Е. Россовским в работе \cite{6}, где исследовалась следующая система управления нейтрального типа:
\begin{equation}
y'(t) + ay'(q^{-1}t) + by(t) + cy(q^{-1}t) = u(t), \quad t>0,
\label{1}
\end{equation}
\begin{equation}
y(0) = y_0 \in \mathbb{R},
\label{2}
\end{equation}
где $a,b,c\in\mathbb{R}$, $q>1$, а $u(t)$ -- управляющее воздействие, которое является вещественнозначной функцией; состояние системы в начальный момент времени задается условием \eqref{2}.

Задача управления формулируется следующим образом: требуется найти $u(t)$, приводящее систему \eqref{1}, \eqref{2} в равновесие $y(t) = 0$ при $t \geqslant T$ для некоторого $T>0$.

Для этого достаточно найти $u(t) \in L_2(0,T)$, приводящее систему в состояние 
\begin{equation}
y\left(t\right)=0, \quad q^{-1}T \leqslant t \leqslant T, 
\label{3}
\end{equation}
а затем сбросить управление, положив $u\left(t\right) \equiv 0$ при $t>T$.  При этом из всех возможных управлений ищется управление, обладающее минимальной энергией
$\|u\|^2_{L_2(0,T)}$.

В результате получается вариационная задача о минимизации квадратичного функционала
\begin{equation}\mathcal{J}\left(y\right)=\displaystyle\int_0^{T}\left(y^{\prime}\left(t\right)+a y^{\prime}\left(q^{-1} t\right) +b y\left(t\right)+c y\left(q^{-1} t\right)\right)^2 d t \longrightarrow \min \label{4}\end{equation}
на множестве функций $y\left(t\right)\in W^1_2[0,T]$, удовлетворяющих краевым условиям  \eqref{2}, \eqref{3}. 

Исследование вариационной задачи \eqref{2}--\eqref{4} включает сведение ее к эквивалентной краевой задаче для функционально-дифференциального уравнения второго порядка с растяжением и сжатием аргумента.

В частности, установлено, что если функция $y\left(t\right)\in W_2^1[0,T]$, удовлетворяющая условиям \eqref{2},\eqref{3}, минимизирует функционал \eqref{4}, то она является решением краевой задачи для уравнения
\begin{multline}
-\left(\left(1+a^2 q\right) y^{\prime}(t)+a y^{\prime}\left(q^{-1} t\right)+a q y^{\prime}(q t)\right)^{\prime}+\left(a b-c q^{-1}\right) y^{\prime}\left(q^{-1} t\right)+ \\
+\left(c q-a b q^2\right) y^{\prime}(q t)+\left(b^2+c^2 q\right) y(t)+b c y\left(q^{-1} t\right)+b c q y(q t)=0, \quad 0<t<q^{-1}T, 
\label{-1}
\end{multline}
при краевых условиях \eqref{2} и \eqref{3}. При этом, поскольку задача  \eqref{2}, \eqref{3}, \eqref{-1} может не иметь решения в $W_2^2\left[0,q^{-1}T\right]$, ее решение является обобщенным в смысле выполнения условия
$$
\left(1+a^2 q\right) y^{\prime}(t)+a y^{\prime}\left(q^{-1} t\right)+a q y^{\prime}(q t)\in W_2^1\left[0,q^{-1}T\right].
$$

Обратное утверждение также верно: если $y\left(t\right)\in W^1_2[0,T]$ является обобщенным решением задачи \eqref{2}, \eqref{3}, \eqref{-1}, то $y$ доставляет минимум функционалу \eqref{4}.

Следующая теорема при предположении $|a|\neq q^{-\frac{1}{2}}$, обеспечивающем коэрцитивность функционала $\mathcal{J}(y)$, устанавливает существование и единственность обобщенного решения краевой задачи \eqref{2}, \eqref{3}, \eqref{-1} и, стало быть, однозначную разрешимость вариационной задачи \eqref{2}--\eqref{4}.

\begin{theorem}[\cite{6}]\label{Th-2} Пусть $|a| \neq q^{-\frac{1}{2}}$. Тогда задача \eqref{2}, \eqref{3}, \eqref{-1}, имеет единственное обобщенное решение $y \in W_2^1[0,T]$. \end{theorem}

В следующем разделе дается постановка вариационной задачи на графе-звезде. Далее, в третьем разделе, следуя общей стратегии для интервала, устанавливается эквивалентность вариационной задачи некоторой краевой задаче для функционально-дифференциальных уравнений второго порядка на графе. В заключительном разделе доказывается однозначная разрешимость обеих задач.

\section{Постановка вариационной задачи на графе-звезде} 

Рассмотрим, изображенный на рисунке \ref{fig1}, граф типа звезды $\Gamma_m$, состоящий из $m$ ребер. Как обычно, под функцией $y$ на графе $\Gamma_m$ будем понимать кортеж $y=[y_1,\dots,y_m]$, в котором компонента $y_j$ определена на ребре $e_j$, т.е. $y_j=y_j\left(t\right)$,  $t\in[0,T_j]$.

Пусть до момента времени $t=T_1>0,$ ассоциированного с единственной внутренней вершиной $v_1$ графа $\Gamma_m$, наша система управления с запаздыванием пропорциональным времени на $\Gamma_m$ описывается уравнением
\begin{equation}\ell_1 y(t):=y_1^\prime(t)+a_1y_1^\prime(q^{-1}t)+b_1y_1(t)+c_1y_1(q^{-1}t)=u_1(t), \quad 0<t<T_1,\label{1.1}\end{equation}
заданным на ребре $e_1$ графа $\Gamma_m$, с начальным условием 
\begin{equation}y_1(0)=y_0. \label{1.5}\end{equation}
При $t=T_1$, т.е. в вершине $v_1$, система разветвляется на $m-1$ независимых параллельных процессов, описываемых уравнениями 
\begin{equation}\begin{array}{c}\ell_j y(t):=y_j^\prime(t)+a_jy_j^\prime(q^{-1}(t-(q-1)T_1))+b_jy_j(t)+c_jy_j(q^{-1}(t-(q-1)T_1))=u_j(t),\\ t>0,\quad j=\overline{2, m},\end{array} \label{1.2}\end{equation}
но имеющих общую историю, определяемую уравнением \eqref{1.1} с начальным условием \eqref{1.5} и условиями
\begin{equation}y_j(t)=y_1(t+T_1), \quad (q^{-1}-1)T_1<t<0, \quad j=\overline{2,m}, \label{1.3}\end{equation}
а также условиями непрерывности в вершине $v_1$, которые в данном случае согласуются с \eqref{1.3} при $t\to-0$: 
\begin{equation}y_j(0)=y_1(T_1), \quad j=\overline{2,m}. \label{1.4}\end{equation}
Как и в случае интервала, мы предполагаем, что $q>1$, $y_0\in\mathbb{R}$ и все $a_j$, $b_j$, $c_j\in\mathbb{R}$.

\begin{figure}[!ht]
\begin{center}
\includegraphics[scale=0.3]{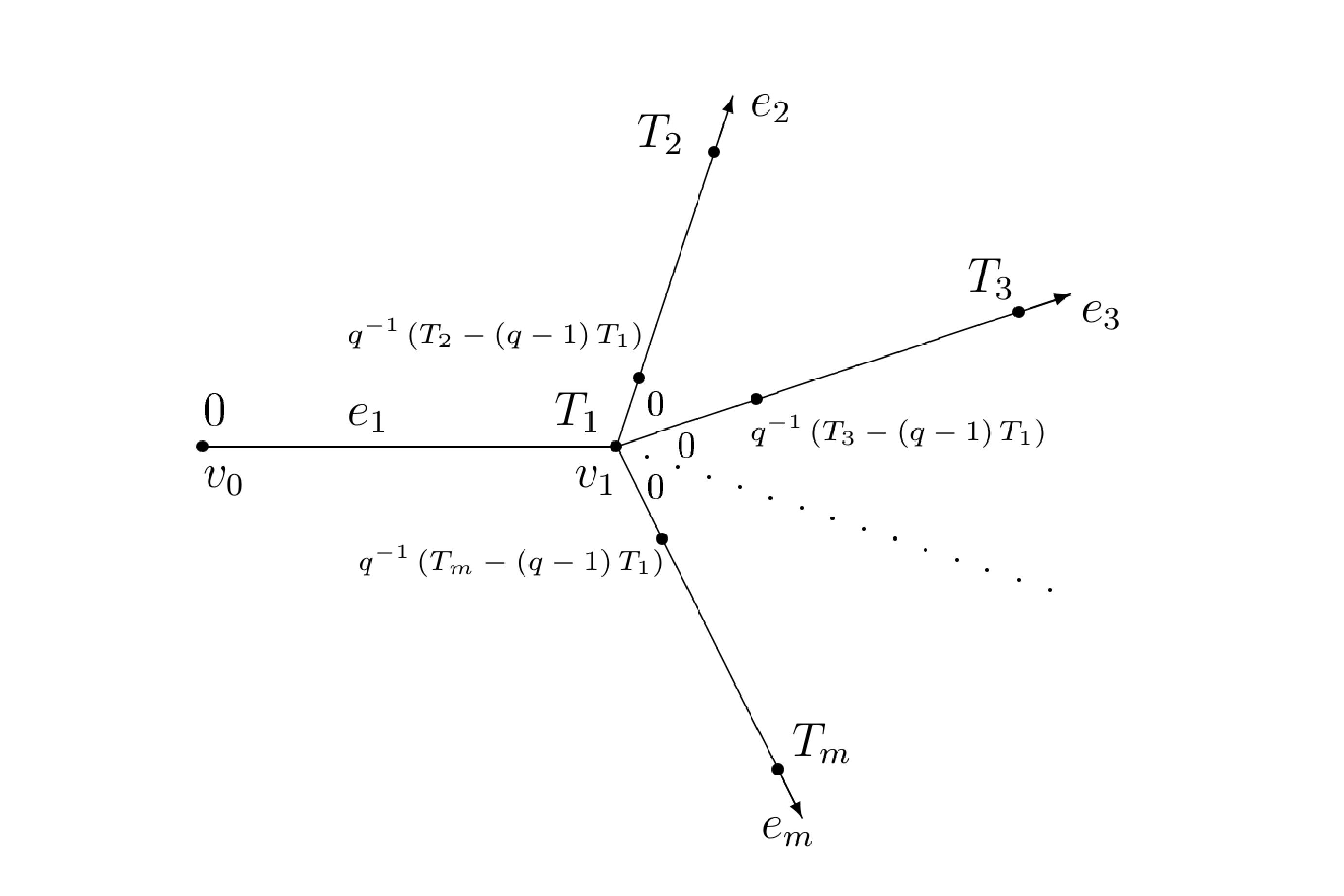}
\end{center}
\caption{\small{Граф $\Gamma_m$}}
\label{fig1}
\end{figure}

В \eqref{1.2} $j$-ое уравнение задано на ребре $e_j$ графа $\Gamma_m$, представляющем собой, вообще говоря бесконечный луч, выходящий из вершины $v_1$. Условия \eqref{1.3} означают, что запаздывание распространяется через вершину $v_1$.

\begin{example}
Пусть $m = 2$, $a := a_1 = a_2$, $b := b_1 = b_2$, $c := c_1 = c_2$, и  
$$
y(t) := 
\begin{cases} 
y_1(t), & 0 \leq t \leq T_1, \\
y_2(t - T_1), & t > T_1, 
\end{cases}
\quad
u(t) := 
\begin{cases} 
u_1(t), & 0 < t < T_1, \\
u_2(t - T_1), & t > T_1.
\end{cases}
$$
Тогда система управления \eqref{1.1}--\eqref{1.4} принимает вид \eqref{1}, \eqref{2}.
\end{example}

Предположим для определенности, что $T_j>\left(q-1\right)T_1$ при всех $j=\overline{2,m}$. Для успокоения системы \eqref{1.1}--\eqref{1.4} сразу при всех сценариях  
нужно привести ее в состояние
\begin{equation}
y_j\left(t\right)=0, \quad q^{-1}\left(T_j-\left(q-1\right)T_1\right)\leqslant t\leqslant T_j, \quad j=\overline{2,m},
\label{1.6}
\end{equation}
выбрав подходящие управления  $u_j\left(t\right)$, $j=\overline{1,m}$. Тогда, положив $u_j\left(t\right)\equiv0$ при $t>T_j$, $j=\overline{2,m}$, будем иметь $y_j(t)=0$ при тех же $t$ и $j$. Другими словами, система будет приведена в равновесие на каждом ребре, выходящем из вершины $v_1$. Поскольку такие $u_j\left(t\right)$ не единственны, будем искать их, минимизируя усилия $\|u_j\|^2_{L_2\left(0,T_j\right)}$.  Кроме того, аналогично тому, как это было сделано в \cite{3} для случая постоянного запаздывания, мы можем регулировать степень участия каждого $\|u_j\|^2_{L_2\left(0,T_j\right)}$ в соответствующем функционале энергии, выбирая определенный положительный вес $\alpha_j$. 

Таким образом, приходим к вариационной задаче 
\begin{equation}
\mathcal{J}\left(y\right)=\sum_{j=1}^m\alpha_j\displaystyle\int_0^{T_j}\left(\ell_j y\left(t\right)\right)^2\,dt\to\min
\label{1.7}
\end{equation}
при условиях \eqref{1.5}, \eqref{1.3}--\eqref{1.6}, где $\alpha_j>0$, $j=\overline{1,m}$, фиксированы.

Для выбора весов $\alpha_j$, $j=\overline{1,m}$, можно применить вероятностный подход в соответствии с интерпретацией системы управления на временн\'ом графе, предложенной в \cite{3}. А именно, для этого нужно положить $\alpha_1=1$, а в качестве $\alpha_j$, $j=\overline{2,m}$, взять вероятности сценариев, задаваемых соответствующими уравнениями в \eqref{1.2}. Тогда $\alpha_2+\ldots+\alpha_m=1$. Последнее тождество также обеспечивает соответствие случаю интервала, если уравнения в \eqref{1.2} не зависят от $j$, т.е. являются искусственными копиями единственного возможного сценария (см. пример 2 в \cite{3}).

Заметим, что условия \eqref{1.3} никаких ограничений на функцию $y=[y_1,\dots,y_m]$ не накладывают. Поэтому условимся, что взятие $\mathcal{J}\left(y\right)$, равно как и $\ell_j y$ при $j=\overline{2,m}$, от какой бы то ни было функции $y$ на $\Gamma_m$ автоматически подразумевает применение условий \eqref{1.3}. Для краткости также введем обозначение $\ell y:=\left[\ell_1y,\dots,\ell_my\right]$.

\section{Сведение к краевой задаче}

Рассмотрим вещественное гильбертово пространство $W_2^k\left(\Gamma_m\right)=\bigoplus_{j=1}^m W_2^k[0,T_j]$ со скалярным произведением $$\left(y,z\right)_{W_2^k\left(\Gamma_m\right)}=\sum_{j=1}^m\left(y_j,z_j\right)_{W_2^k[0,T_j]},$$ где $y=[y_1,\dots,y_m]$, $z=[z_1,\dots,z_m]$, $\left(f,g\right)_{W_2^k[a,b]}=\sum_{\nu=0}^{k} \left(f^{\left(\nu\right)},g^{\left(\nu\right)}\right)_{L_2\left(a,b\right)}$ -- скалярное произведение в $W_2^k [a,b]$, а $\left(\cdot,\cdot\right)_{L_2\left(a,b\right)}$ -- скалярное произведение в $L_2\left(a,b\right)$.

Обозначим через $\mathcal{W}$ замкнутое подпространство $W_2^1\left(\Gamma_m\right)$, состоящее из всех наборов $[y_1,\dots,y_m]$, удовлетворяющих условиями \eqref{1.4}, \eqref{1.6} и $y_1\left(0\right)=0$.

Также введем пространство $W_2^k\left(\widetilde{\Gamma}_m\right)=W_2^k[0,T_1]\oplus\displaystyle\bigoplus_{j=2}^m W_2^k[0,q^{-1}\left(T_j-\left(q-1\right)T_1\right)]$.

\begin{lemma}\label{L1}Если $y\in W_2^1\left(\Gamma_m\right)$ является решением вариационной задачи \eqref{1.5}, \eqref{1.3}--\eqref{1.7}, то 
\begin{equation}
B\left(y, w\right):=\sum_{j=1}^m \alpha_j\displaystyle\int_0^{T_j} \ell_j y\left(t\right)\ell_j  w\left(t\right) dt =0 \quad \forall w\in\mathcal{W}.
\label{2.1}
\end{equation}
Обратно, если для некоторого $y\in W_2^1\left(\Gamma_m\right)$ выполняется \eqref{1.5}, \eqref{1.4}, \eqref{1.6} и \eqref{2.1}, то $y$ является решением задачи \eqref{1.5}, \eqref{1.3}--\eqref{1.7}.
\end{lemma}

\begin{proof}
Пусть $y\in W_2^1\left(\Gamma_m\right)$ -- решение задачи \eqref{1.5}, \eqref{1.3}--\eqref{1.7}. Тогда для произвольной фиксированной функции $ w\in\mathcal{W}$ сумма $y+s w$ принадлежит $W_2^1\left(\Gamma_m\right)$ при любом $s\in\mathbb{R}$ и удовлетворяет условиям  \eqref{1.5}, \eqref{1.4}, \eqref{1.6}. Нетрудно видеть, что $$\mathcal{J}\left(y+s w\right) = \mathcal{J}\left(y\right)+2sB\left(y, w\right)+s^2\mathcal{J}\left( w\right).$$ Так как  $\mathcal{J}\left(y+s w\right)\geqslant\mathcal{J}\left(y\right)$ для всех $s\in\mathbb{R}$, то выполняется \eqref{2.1}.

Обратно, пусть  $y\in W_2^1\left(\Gamma_m\right)$ удовлетворяет условиям \eqref{1.5}, \eqref{1.4} и \eqref{1.6}. Тогда \eqref{2.1} влечет
$$\mathcal{J}\left(y+ w\right)=\mathcal{J}\left(y\right)+2B\left(y, w\right)+\mathcal{J}\left( w\right)\geqslant\mathcal{J}\left(y\right) \quad \forall  w\in\mathcal{W}.$$
Таким образом, $y$ доставляет минимум функционалу \eqref{1.7} при условиях \eqref{1.5}, \eqref{1.4} и \eqref{1.6}. 
\end{proof}

Преобразуем \eqref{2.1}, сделав замену переменных в членах, содержащих $w_j\left(q^{-1}\left(t-\left(q-1\right)T_1\right)\right)$ и $w_1\left(q^{-1}t\right)$. 
В результате выражение для $B\left(y, w\right)$ примет вид
\begin{multline*}
B\left(y, w\right) = \displaystyle\sum_{j=1}^m\alpha_j\displaystyle\int_0^{T_j}\ell_jy\left(t\right)\left( w_j^\prime\left(t\right)+b_jw_j\left(t\right)\right)dt +\\+ 
q\left(\alpha_1\displaystyle\int_0^{q^{-1}T_1}\ell_1y\left(qt\right) \left(a_1w_1^\prime\left(t\right)+c_1w_1\left(t\right)\right)dt \right.+\\+\left. \displaystyle\sum_{j=2}^m\alpha_j 
\displaystyle\int_{\left(q^{-1}-1\right)T_1}^{q^{-1}\left(T_j-\left(q-1\right)T_1\right)} \ell_j y\left(qt+\left(q-1\right)T_1\right) \left(a_jw_j^\prime\left(t\right)+c_jw_j\left(t\right)\right)dt \right).
\end{multline*}
Применяя \eqref{1.3} к $ w=\left[ w_1,\dots, w_m\right]\in\mathcal{W}$, можно представить
\begin{multline*}\displaystyle\int_{\left(q^{-1}-1\right)T_1}^{q^{-1}\left(T_j-\left(q-1\right)T_1\right)} \ell_j y\left(qt+\left(q-1\right)T_1\right) \left( a_jw_j^\prime\left(t\right)+b_jw_j\left(t\right)\right)dt =
\\ 
=\displaystyle\int_{0}^{q^{-1}\left(T_j-\left(q-1\right)T_1\right)} \ell_j y\left(qt+\left(q-1\right)T_1\right) \left(a_jw_j^\prime\left(t\right)+c_jw_j\left(t\right)\right)
dt + \\  +\displaystyle\int_{q^{-1}T_1}^{T_1} \ell_j y\left(qt-T_1\right) \left(a_jw_1^\prime\left(t\right)+c_jw_1\left(t\right)\right)dt, \quad j=\overline{2,m}.\end{multline*}
Тогда перепишем \eqref{2.1} в эквивалентном виде
\begin{multline}
B\left(y, w\right)=\displaystyle\sum_{j=1}^m\left(\displaystyle\int_0^{T_j}
\left(\alpha_j\ell_{j}y\left(t\right)+\widetilde{\ell}_{1,j}y\left(t\right)\right)
 w_j^\prime\left(t\right)dt+\right.
\\
+\left.\displaystyle\int_0^{T_j}\left(\alpha_jb_j\ell_{j}y\left(t\right)+\widetilde{\ell}_{0,j}y\left(t\right)\right) w_j\left(t\right)dt\right)=0\quad\forall w\in\mathcal{W},
\label{2.2}
\end{multline}
где
\begin{equation}
\begin{aligned}
\widetilde{\ell}_{\nu,1} y\left(t\right) &= \left\{\begin{array}{cc} q \alpha_1\theta_{\nu,1}\ell_1 y\left(qt\right), & 0<t<q^{-1}T_1,\\[3mm] q \displaystyle\sum_{k=2}^m \alpha_k \theta_{\nu,k}\ell_k y\left(qt-T_1\right), & q^{-1}T_1< t<T_1, \end{array}\right. 
\\
\widetilde{\ell}_{\nu,j} y\left(t\right) &= \left\{\begin{array}{cc} q \alpha_j \theta_{\nu,j}\ell_j y\left(qt+\left(q-1\right)T_1\right), & 0<t<l_j,\\[3mm]\displaystyle 0, & l_j< t<T_j,  \end{array}\right. \quad j=\overline{2,m},
\end{aligned} 
\label{2.4}
\end{equation}
при $\nu\in\{0,1\}$ и
$$
\theta_{\nu,j} = \begin{cases} 
a_j, &  \nu = 1, \\
c_j, &  \nu = 0,
\end{cases}  \quad j=\overline{1, m}.
$$

Обозначим через $\mathcal{B}$ краевую задачу для функционально-дифференциальных уравнений второго порядка
\begin{equation}
\mathcal{L}_j y\left(t\right):=-\left(\widehat{\ell}_jy\left(t\right)\right)^\prime
+\alpha_j b_j \ell_j y\left(t\right) + \widetilde{\ell}_{0,j}y\left(t\right)=0,\quad 0<t<l_j,\quad j=\overline{1,m},
\label{2.5}
\end{equation}
при условиях \eqref{1.5}, \eqref{1.3}--\eqref{1.6} и условии типа Кирхгофа
\begin{equation}
\widehat{\ell}_1y\left(T_1\right)=\displaystyle\sum_{j=2}^m\widehat{\ell}_jy\left(0\right),
\label{2.6}
\end{equation}
где $\widetilde{\ell}_{\nu,j}$, $\nu\in\{0,1\}$, определены в \eqref{2.4}, а выражения $\widehat{\ell}_jy\left(t\right)$ имеют вид
\begin{equation}
\widehat{\ell}_jy\left(t\right):= \alpha_j\ell_j y\left(t\right)+\widetilde{\ell}_{1,j}y\left(t\right), \quad j=\overline{1, m},
\label{2.7}
\end{equation}
тогда как
\begin{equation}
l_1:=T_1, \quad l_j:=q^{-1}\left(T_j-\left(q-1\right)T_1\right), \quad j=\overline{2,m}.
\label{2.8}
\end{equation}

\begin{definition}
Функцию $y=[y_1,\dots,y_m]\in W_2^1({\Gamma}_m)$ назовем обобщенным решением краевой задачи $\mathcal{B}$, если $\widehat{\ell}_jy\left(t\right)\in W_2^1\left[0,l_j\right]$, $j=\overline{1, m}$, а функции $y_j$, $j=\overline{1, m}$, удовлетворяют уравнениям  \eqref{2.5} и условиям \eqref{1.5}, \eqref{1.3}--\eqref{1.6}, \eqref{2.6}.
\end{definition}

Имеет место следующее утверждение.

\begin{lemma}\label{L2}Если $y\in W_2^1\left(\Gamma_m\right)$ удовлетворяет условиям \eqref{1.5}, \eqref{1.4}, \eqref{1.6} и \eqref{2.1}, то $y$ является обобщенным решением краевой задачи $\mathcal{B}$. Обратно, любое обобщенное решение краевой задачи $\mathcal{B}$, подчиняется условию \eqref{2.1}.
\end{lemma}

\begin{proof}
Пусть $y\in W_2^1\left(\Gamma_m\right)$  и удовлетворяет условиям \eqref{1.5}, \eqref{1.4}, \eqref{1.6} и \eqref{2.1}. Учитывая, что \eqref{2.1} эквивалентно \eqref{2.2} и применяя лемму 2 из \cite{3} к \eqref{2.2} вместе с \eqref{2.8}, получаем, что $\widehat{\ell}_jy\left(t\right)\in W_2^1\left[0,l_j\right]$, $j=\overline{1, m}$, и выполняется \eqref{2.6}.
Кроме того, интегрируя по частям в \eqref{2.2} и используя  \eqref{1.4}, \eqref{1.6},  будем иметь 
\begin{equation}
B\left(y, w\right)=  w_1\left(T_1\right)\left(\widehat{\ell}_1y\left(T_1\right)-\displaystyle\sum_{j=2}^m\widehat{\ell}_jy\left(0\right)\right)+ \displaystyle\sum_{j=1}^m\displaystyle\int_0^{l_j}\mathcal{L}_j y\left(t\right) w_j\left(t\right)dt=0.
\label{2.9}
\end{equation}
В силу \eqref{2.6}, а также произвольности $ w_j$, из \eqref{2.9} получаем \eqref{2.5}.

Обратно, пусть $y$ --  обобщенное решение задачи $\mathcal{B}$. Тогда левое равенство в \eqref{2.9} дает \eqref{2.1}.
\end{proof}

Объединив леммы \ref{L1} и \ref{L2}, получаем следующий результат.

\begin{theorem}\label{Th1}Функция $y\in W_2^1\left(\Gamma_m\right)$ является решением вариационной задачи \eqref{1.5}, \eqref{1.3}--\eqref{1.7} тогда и только тогда, когда $y$ является обобщенным решением краевой задачи $\mathcal{B}$.\end{theorem}

\section{Однозначная разрешимость}  

В данном разделе устанавливается однозначная разрешимость краевой задачи $\mathcal{B}$, а согласно теореме \ref{Th1} -- и вариационной задачи \eqref{1.5}, \eqref{1.3}--\eqref{1.7}.

Введем обозначения
\begin{equation}
\ell_1^0 y\left(t\right)=y_1^\prime\left(t\right)+a_1y_1^\prime\left(q^{-1}t\right),\;\;\quad\ell_j^0 y\left(t\right)=y_j^\prime\left(t\right)+a_jy_j^\prime\left(q^{-1}\left(t-\left(q-1\right)T_1\right)\right),\;\;\; \quad j=\overline{2, m},
\label{3.5}
\end{equation}
\begin{equation}
\ell_1^1 y\left(t\right)=b_1y_1\left(t\right)+c_1y_1\left(q^{-1}t\right),\quad
\ell_j^1 y\left(t\right)=b_jy_j\left(t\right)+c_jy_j\left(q^{-1}\left(t-\left(q-1\right)T_1\right)\right), \quad j=\overline{2, m},
\label{3.5.2}
\end{equation}
$$\mathcal{J}_\nu\left(y\right)=\sum_{j=1}^m\displaystyle\int_0^{T_j}\left|\ell_j^\nu y\left(t\right)\right|^2dt,\quad \nu=0,1,$$
где автоматически предполагается \eqref{1.3}.

\begin{lemma}\label{L3} Существуют $C_0$ и $C_1$ такие, что 
\begin{equation}
\mathcal{J}(w)\leqslant C_0 \| w\|^2_{W^1_2\left(\widetilde{\Gamma}_m\right)},\quad\quad \mathcal{J}_1(w)\leqslant C_1 \| w\|^2_{L_2\left(\widetilde{\Gamma}_m\right)} \quad\quad\forall w\in\mathcal{W}.
\label{3.1}
\end{equation}
\end{lemma}
\begin{proof}
Пусть $w\in\mathcal{W}$. Используя \eqref{3.5}, \eqref{3.5.2} и неравенство
\begin{equation}
\left(\gamma_1+\dots+\gamma_n\right)^2 \leqslant n\left(\gamma_1^2+\dots+\gamma_n^2\right),\quad  \gamma_1,\dots,\gamma_n\in{\mathbb R},
\label{3.2}
\end{equation}
для $n=2$, получаем
$$
\mathcal{J}_0(w) \leqslant 
2 \sum_{j=1}^m \|w_j^\prime\|^2_{L_2(0,T)}   
+ 
2\left(a_1^2\int_{0}^{T_1}\left(w_1^\prime\left(q^{-1}t\right)\right)^2dt + 
\sum_{j=2}^m a_j^2 \int_{0}^{T_j}  \left(w_j^\prime\left(q^{-1}\left(t-\left(q-1\right)T_1\right)\right)\right)^2  dt\right),
$$
$$
\mathcal{J}_1(w) \leqslant 
2 \sum_{j=1}^m b_j^2 \|w_j\|^2_{L_2(0,T)}
+
2\left(c_1^2\int_{0}^{T_1}w_1^2\left(q^{-1}t\right)dt + \sum_{j=2}^m c_j^2 \int_{0}^{T_j}  w_j^2\left(q^{-1}\left(t-\left(q-1\right)T_1\right)\right)  dt \right).\;\;\;\;\;\;
$$
Учитывая \eqref{1.3} применительно к $w$, для $\nu=0,1$ вычисляем
$$
\int_{0}^{T_1} 
\left(w_1^{(\nu)}\left(q^{-1}t\right)\right)^2
dt = q\| w_1^{(\nu)}\|^2_{L_2\left(0,q^{-1}T_1\right)},
$$
$$
\int_{0}^{T_j} 
\left(
w_j^{(\nu)}\left(q^{-1}\left(t-\left(q-1\right)T_1\right)\right)
\right)^2
dt  =q\| w_1^{(\nu)}\|^2_{L_2\left(q^{-1}T_1,T_1\right)} +
q\| w_j^{(\nu)}\|^2_{L_2\left(0,q^{-1}\left(T_j-\left(q-1\right)T_1\right)\right)}, \quad j=\overline{2,m}.
$$
В частности, это дает вторую оценку в \eqref{3.1}. 
Наконец, аналогично лемме 5 в \cite{3} нетрудно показать, что $\|w^\prime\|_{L_2\left(\Gamma_m\right)}$ порождает норму в $\mathcal{W}$, эквивалентную норме $\|w\|_{W^1_2\left(\widetilde{\Gamma}_m\right)}$.
Таким образом, используя оценку
$$
\mathcal{J}(w)\leqslant 2\widetilde{\alpha}_1\left(\mathcal{J}_0(w)+\mathcal{J}_1(w)\right),\quad \widetilde{\alpha}_1=\max\left\{\alpha_1,\dots,\alpha_m\right\},
$$ 
получаем первое неравенство в \eqref{3.1}. 
\end{proof}

\begin{lemma}\label{L4} Пусть $|a_1|\neq q^{-\frac{1}{2}}$ и ${|a_2|+\dots+|a_m|>0}$.
Тогда существует $C_2>0$ такое, что 
\begin{equation}
\mathcal{J}_0\left(w\right)\geqslant C_2\|w\|^2_{W_2^1\left(\widetilde{\Gamma}_m\right)} \quad\forall w\in\mathcal{W}.
\label{3.6}
\end{equation}
\end{lemma}
\begin{proof}
Предположим от противного, что найдутся $w_{\left(s\right)}=\left[w_{\left(s\right),1},\dots,w_{\left(s\right),m}\right]\in\mathcal{W}$ при $s\in\mathbb{N}$, такие что 
$
\|w_{\left(s\right)}\|_{W_2^1\left(\widetilde{\Gamma}_m\right)}=1
$
и 
\begin{equation}
\mathcal{J}_0\left(w_{\left(s\right)}\right)\leq\frac{1}{s}, \quad s\in\mathbb{N}.
\label{3.9}
\end{equation}

Используя первое выражение в \eqref{3.5} и \eqref{3.9}, для всех $s\in\mathbb{N}$ получаем 
\begin{equation}
\int_0^{T_1}\left(w^\prime_{\left(s\right),1}\left(t\right)\right)^2dt+2a_1\int_0^{T_1}w^\prime_{\left(s\right),1}\left(t\right)w^\prime_{\left(s\right),1}\left(q^{-1}t\right)dt + a_1^2\int_0^{T_1}\left(w^\prime_{\left(s\right),1}\left(q^{-1}t\right)\right)^2dt \leq \frac{1}{s}.
\label{3.-5}
\end{equation}
Применим неравенство Коши-Буняковского к среднему интегралу:
$$
a_1\int_0^{T_1}w^\prime_{\left(s\right),1}\left(t\right)w^\prime_{\left(s\right),1}\left(q^{-1}t\right)dt \geqslant -|a_1|\sqrt{q}\|w^\prime_{\left(s\right),1}\|_{L_2\left(0,T_1\right)}^2.
$$
Подставляя эту оценку в \eqref{3.-5}, имеем
\begin{equation}
\left(1-|a_1|\sqrt{q}\right)^2\|w^{\prime}_{\left(s\right),1}\|_{L_2\left(0,T_1\right)}^2\leq\frac{1}{s}+ a_1^2q\|w^{\prime}_{\left(s\right),1}\|_{L_2\left(q^{-1}T_1,T_1\right)}^2.
\label{3.-6}
\end{equation}

С другой стороны, согласно  \eqref{3.5}, для всех $j\in\overline{2, m}$, таких что  $a_j\neq 0$, почти всюду на $\left(0,T_j\right)$ справедливы оценки
\begin{equation}
\left|w_j^\prime\left(q^{-1}\left(t-\left(q-1\right)T_1\right)\right)\right|\leq \left|a_j\right|^{-1}\left|\ell_j^0 w\left(t\right)\right|+\left|a_j\right|^{-1}\left|w_j^\prime\left(t\right)\right|.
\label{3.7}
\end{equation}

Используя \eqref{3.5}, \eqref{3.2}, \eqref{3.9} и \eqref{3.7}, при всех $s\in\mathbb{N}$ и $j=\overline{2, m}$ получаем оценки
$$
\begin{cases}
\|w_{\left(s\right),j}^\prime\|^2_{L_2\left(0,T_j\right)} \leq\frac{1}{s},       & \quad \text{если }  a_j= 0,
\\
\begin{matrix}
\|w_{\left(s\right),j}^\prime\|^2_{L_2\left(\left(q^{-1}-1\right)T_1,0\right)}\leq\frac{2}{qa_j^2s}+\frac{2}{qa_j^2} \|w_{\left(s\right),j}^\prime\|^2_{L_2\left(0,\left(q-1\right)T_1\right)},  
\\
\|w_{\left(s\right),j}^\prime\|^2_{L_2\left(T_{j,k+1},T_{j,k}\right)}\leq\frac{2}{qa_j^2s}+\frac{2}{qa_j^2} \|w_{\left(s\right),j}^\prime\|^2_{L_2\left(T_{j,k},T_{j,k-1}\right)},
 \end{matrix}
& \quad \text{если }  a_j\neq 0,
\end{cases}
$$
где $k\in\left\{n\in\mathbb{N}:T_j > \left(q^{n} - 1\right)T_1\right\}$ и
$
T_{j,k}:= q^{-k}\left(T_j-\left(q^{k}-1\right)T_1\right).
$
Учитывая, что
$$
\|w_{(s),j}^\prime\|_{L_2\left(q^{-1}\left(T_j-(q-1)T_1\right),T_j\right)}=0, \quad j=\overline{2,m},
$$
из полученных неравенств следует
$\|w_{\left(s\right),j}^\prime\|_{L_2\left(0,T_j\right)}\to 0$  при $s\to\infty$ для $j=\overline{2,m}$. 
Более того, если  $a_j\neq 0$, то для таких $j$ дополнительно имеем $\|w_{\left(s\right),j}^\prime\|_{L_2\left(\left(q^{-1}-1\right)T_1,0\right)}\to 0$  при $s\to\infty$. Последнее, вместе с тождествами 
$$
\|w_{\left(s\right),j}^\prime\|^2_{L_2\left(
\left(q^{-1}-1\right)T_1,0
\right)}=\|w_{\left(s\right),1}^\prime\|^2_{L_2\left(q^{-1}T_1,T_1\right)},\quad j=\overline{2,m},
$$ 
и оценками \eqref{3.-6}, позволяет заключить, что $\|w_{\left(s\right),1}^\prime\|_{L_2\left(0,T_1\right)}\to 0$  при $s\to\infty$.

В итоге имеем $\|w_{\left(s\right)}^\prime\|_{L_2\left(\Gamma_m\right)}\to 0$  при $s\to\infty$, где $w^\prime=\left[w_{1}^\prime,\dots,w_{m}^\prime\right]$.
Поскольку $\|w^\prime\|_{L_2\left(\Gamma_m\right)}$ порождает норму в $\mathcal{W}$, эквивалентную норме $\|w\|_{W^1_2\left(\widetilde{\Gamma}_m\right)}$, то и $\|w_{\left(s\right)}\|_{W^1_2\left(\widetilde{\Gamma}_m\right)}\to 0$ при $s\to\infty$. Последнее противоречит $\|w_{\left(s\right)}\|_{W_2^1\left(\widetilde{\Gamma}_m\right)}=1$.

\end{proof}

\begin{lemma}\label{L5} Пусть $|a_1|\neq q^{-\frac{1}{2}}$ и ${|a_2|+\dots+|a_m|>0}$. Тогда существует $C_3>0$ такое, что 
$$\mathcal{J}\left( w\right)\geqslant C_3\| w\|^2_{W^1_2\left(\widetilde{\Gamma}_m\right)} \quad\forall w\in\mathcal{W}.$$
\end{lemma}

\begin{proof} 
Снова предположим от противного, что найдутся $w_{\left(s\right)}\in\mathcal{W}$, $s\in\mathbb{N}$, такие, что
$
\| w_{\left(s\right)}\|_{W^1_2\left(\widetilde{\Gamma}_m\right)}=1,
$
но теперь
\begin{equation}
\mathcal{J}\left( w_{\left(s\right)}\right)\leqslant\frac{1}{s}, \quad s\in\mathbb{N}.
\label{3.13}
\end{equation}

Из неравенства
$$ \widetilde{\alpha}_2\mathcal{J}_0\left(w\right)\leq 2\mathcal{J}\left(w\right)+2\widetilde{\alpha}_1\mathcal{J}_1\left(w\right),\quad \widetilde{\alpha}_1=\max\left\{\alpha_1,\dots,\alpha_m\right\},\quad\widetilde{\alpha}_2=\min\left\{\alpha_1,\dots,\alpha_m\right\}, $$
совместно с оценками  \eqref{3.1} и \eqref{3.6} получаем
\begin{equation}
\frac{\widetilde{\alpha}_2 C_2}{2}\| w\|^2_{W^1_2\left(\widetilde{\Gamma}_m\right)}\leqslant\mathcal{J}\left( w\right)+\widetilde{\alpha}_1 C_1\| w\|^2_{L_2\left(\Gamma_m\right)}.
\label{3.18}
\end{equation}

В силу компактности вложения $W^1_2\left(\Gamma_m\right)$ в $L_2\left(\Gamma_m\right)$ найдется подпоследовательность $\{ w_{\left(n_k\right)}\}_{k\in\mathbb{N}}$ фундаментальная в $L_2\left(\Gamma_m\right)$.
Тогда неравенство \eqref{3.18} дает
$$
\frac{\widetilde{\alpha}_2 C_2}{2}\| w_{\left(s_k\right)}- w_{\left(s_l\right)}\|^2_{W^1_2\left(\widetilde{\Gamma}_m\right)}\leqslant\mathcal{J}\left( w_{\left(s_k\right)}- w_{\left(s_l\right)}\right)+\widetilde{\alpha}_1C_1\| w_{\left(s_k\right)}- w_{\left(s_l\right)}\|^2_{L_2\left(\Gamma_m\right)}.
$$
Кроме того, используя \eqref{3.2} при $n=2$ и \eqref{3.13}, имеем
$$\mathcal{J}\left( w_{\left(s_k\right)}- w_{\left(s_l\right)}\right)\leqslant\frac{2}{s_k}+\frac{2}{s_l}.$$
Таким образом, последовательность $\{ w_{\left(s_k\right)}\}_{k\in\mathbb{N}}$ является фундаментальной в $\mathcal{W}$ и сходится к некоторой функции $ w_{\left(0\right)}\in\mathcal{W}$.

В силу леммы \ref{L3} сходимость $ w_{\left(s_k\right)}$ к $w_{\left(0\right)}$ в $\mathcal{W}$ влечет сходимость $\ell w_{\left(s_k\right)}$ к $\ell w_{\left(0\right)}$ в $L_2\left(\Gamma_m\right)$. 
Следовательно, в силу \eqref{3.13} будем иметь
$$\|\ell w_{\left(0\right)}\|^2_{L_2\left(\Gamma_m\right)}=\displaystyle\lim_{k \to \infty}\|\ell w_{\left(s_k\right)}\|^2_{L_2\left(\Gamma_m\right)} =\displaystyle\lim_{k \to \infty}\mathcal{J}\left(w_{\left(s_k\right)}\right)=0,$$
т.е. $w_{\left(0\right)}\in\mathcal{W}$ удовлетворяет уравнениям
$$
w_{\left(0\right),1}^\prime(t)+a_1w_{\left(0\right),1}^\prime(q^{-1}t)+b_1w_{\left(0\right),1}(t)+c_1w_{\left(0\right),1}(q^{-1}t)=0, 
$$
$$
w_{\left(0\right),j}^\prime(t)+a_jw_{\left(0\right),j}^\prime(q^{-1}(t-(q-1)T_1))+b_jw_{\left(0\right),j}(t)+c_jw_{\left(0\right),j}(q^{-1}(t-(q-1)T_1))=0, \; \; j=\overline{2, m}.
$$

Если для некоторого $j\in\overline{1,m}$ выполняется $a_j=c_j=0$, то уравнение для $w_{\left(0\right),j}$ сводится к обычному дифференциальному уравнению, решение которого, с учетом принадлежности $w_{(0)}$ классу $\mathcal{W}$,  дает $w_{\left(0\right),j}(t)=0$ при $0<t<T_j$.

Рассмотрим $j\in\overline{1,m}$, при которых $a_j^2+c_j^2\neq 0$. Поскольку $w_{\left(0\right),j}\left(t\right)=0$, $j=\overline{2, m}$ при 
$t~\geqslant~q^{-1}\left(T_j-\left(q-1\right)T_1\right)$ соответствующие уравнения 
на интервале $\left(q^{-1}\left(T_j-\left(q-1\right)T_1\right),T_j\right)$  принимают вид
$$a_jw_{\left(0\right),j}^\prime(q^{-1}(t-(q-1)T_1))+c_jw_{\left(0\right),j}(q^{-1}(t-(q-1)T_1))=0,\quad j=\overline{2, m},$$
или
$$a_jw_{\left(0\right),j}^\prime(t)+c_jw_{\left(0\right),j}(t)=0,\quad q^{-2}\left(T_j-\left(q^2-1\right)T_1\right)<t<q^{-1}\left(T_j-\left(q-1\right)T_1\right), \quad j=\overline{2, m}.$$
Отсюда с учетом $w_{\left(0\right),j}\left(q^{-1}\left(T_j-\left(q-1\right)T_1\right)\right)=0$ следует, что $w_{\left(0\right),j}(t)=0$, $j=\overline{2, m}$, при $t~\geqslant~q^{-2}\left(T_j-\left(q^2-1\right)T_1\right)$. 
Продолжая дальше влево аналогичным образом, получаем $w_{\left(0\right),j}(t)=0$ при $\left(q^{-1}-1\right)<t<T_j$, $j=\overline{2, m}$.
Учитывая, что $w_{(0),j}(t)=w_{(0),1}(t+T_1)$ при $(q^{-1}-1)T_1<t<0$, из чего следует $w_{(0),1}(t)=0$ для $q^{-1}T_1<t<T_1$, аналогичными рассуждениями можем получить $w_{(0),1}(t)=0$ для $0<t<T_1$.

В итоге имеем $w_{\left(0\right)}(t)=0$, что противоречит $\| w_{\left(0\right)}\|_{W^1_2\left(\widetilde{\Gamma}_m\right)}=1$.
\end{proof}

Следующая теорема является основным результатом данного раздела.

\begin{theorem}\label{Lednov_T2}
Пусть $|a_1|\neq q^{-\frac{1}{2}}$ и ${|a_2|+\dots+|a_m|>0}$. Тогда краевая задача $\mathcal{B}$ имеет единственное обобщенное решение $y\in W^1_2\left(\Gamma_m\right)$. Кроме того, существует $C$ такое, что 
\begin{equation}\|y\|_{W^1_2\left(\Gamma_m\right)}\leqslant C|y_0|.\label{3.19}\end{equation}
\end{theorem}

\begin{proof}
Рассмотрим функцию $\Phi=\left[\Phi_1,\dots,\Phi_m\right]\in W^1_2\left(\Gamma_m\right)$ такую, что 
\[
\Phi_1\left(t\right) = \begin{cases} 
y_0\left(1-\frac{qt}{T_1}\right), & 0 \leqslant t < q^{-1}T_1, \\
0, & q^{-1}T_1 \leqslant t \leqslant T_1,
\end{cases}
\quad \Phi_j\left(t\right) \equiv 0, \quad j = 2, \dots, m.
\]
В силу леммы \ref{L2}, для того чтобы функция $y\in W^1_2\left(\Gamma_m\right)$, удовлетворяющая условиям \eqref{1.5}, \eqref{1.4}, \eqref{1.6}, была решением краевой задачи $\mathcal{B}$, необходимо и достаточно, чтобы выполнялось условие \eqref{2.1}. Другими словами, $y$ является решением краевой задачи $\mathcal{B}$,
если и только если $x:=y-\Phi\in\mathcal{W}$ и
\begin{equation}B\left(\Phi, w\right)+B\left(x, w\right)=0 \quad\forall w\in\mathcal{W}.\label{3.20}\end{equation}

Так как  $B\left( w, w\right)=\mathcal{J}\left( w\right)$, то в силу леммы \ref{L5} билинейная форма $\left(\cdot,\cdot\right)_{\mathcal{W}}:= B\left( \cdot,\cdot \right)$ является скалярным произведением в $\mathcal{W}$. Кроме того, справедлива оценка
\begin{equation}\left|B\left(\Phi, w\right)\right|= \alpha_1\left|\displaystyle\int_{0}^{T_1}\ell_1\Phi\left(t\right)\ell_1 w\left(t\right) dt \right| \leqslant M|y_0| \| w\|_{\mathcal{W}},\label{3.21}\end{equation}
где $\| w\|_{\mathcal{W}}=\sqrt{\left( w, w\right)_{\mathcal{W}}}$. Таким образом, по теореме Рисса об общем виде линейного непрерывного функционала в гильбертовом пространстве существует единственная функция $x\in\mathcal{W}$ такая, что выполняется \eqref{3.20}. Следовательно, краевая задача $\mathcal{B}$ имеет единственное решение $y=\Phi+x$.
Кроме того, согласно \eqref{3.20} и \eqref{3.21} имеем
$$
\|x\|_{\mathcal{W}}\leq M|y_0|,
$$
что позволяет получить оценку \eqref{3.19}.
\end{proof}

\begin{remark}
Утверждение теоремы \ref{Lednov_T2} может быть дополнено следующим результатом из \cite{98}. А именно: в случае, когда система \eqref{1.1}--\eqref{1.4} имеет запаздывающий тип (т.е. $a_j=0$ для всех $j=\overline{1,m}$), краевая задача $\mathcal{B}$ имеет единственное решение; при этом $y\in W^1_2(\Gamma_m)\cap W^2_2(\widetilde{\Gamma}_m)$ и также справедлива априорная оценка \eqref{3.19}.
\end{remark}

\end{document}